\documentclass[final]{siamltex}

\usepackage{amsmath}
\usepackage{amssymb}
\usepackage{graphicx}
\usepackage{color}

\newcommand{\calL}{\mathcal{L}}

\newcommand{\bbR}{\mathbb{R}}
\newcommand{\bbC}{\mathbb{C}}

\newcommand{\ve}{\varepsilon}

\renewcommand{\Re}{\mathfrak{Re}}

\newtheorem{example}{Example}
\newtheorem{remark}{Remark}

\title{Pseudospectral bounds on transient growth
 for higher order and constant delay differential 
 equations\thanks{This material is based upon work
 supported by the National Science Foundation under
 Grant No. DMS-1620038.
}}

\author{
Amanda Hood\thanks{Center for Applied Mathematics, Cornell
    University, Ithaca, NY 14850 ({\tt ah576@cornell.edu})}
\and
David Bindel\thanks{Department of Computer Science, Cornell
    University, Ithaca, NY 14850 ({\tt bindel@cs.cornell.edu})}
}

\begin{document}

\maketitle

\begin{abstract}
 Asymptotic dynamics of ordinary differential equations
 (ODEs)
 are commonly understood by looking at eigenvalues of
 a matrix, and transient dynamics can be bounded above
 and below by
 considering the corresponding pseudospectra. While
 asymptotics for other classes of differential equations
 have been studied using eigenvalues of a (nonlinear)
 matrix-valued function, there are no analogous
 pseudospectral bounds on transient growth. In this
 paper,
 we propose extensions of the pseudospectral results
 for ODEs first to higher order ODEs and then to
 delay differential equations (DDEs) with constant delay.
 Results are illustrated with a discretized partial
 delay differential equation and a model of a semiconductor
 laser with phase-conjugate feedback.
\end{abstract}



\pagestyle{myheadings}
\thispagestyle{plain}
\markboth{A. Hood and D. Bindel}
{Pseudospectral bounds on transient growth
       for higher order and constant delay DEs}

\section{Introduction}
\label{sec-intro}

Nonlinear differential equations are often used to
model economic~\cite{bischi2009nonlinear},
biological~\cite{Kubiaczyk2002},
chemical~\cite{lehman1994recycle},
and physical~\cite{gray1994laser} systems.
Often an equilibrium
solution is of interest, and since the equilibrium
will not actually be achieved in practice,
the behavior of nearby solutions is studied.
To make analysis
of this problem more tractable,
one commonly analyzes stability of the linearized
dynamics near the equilibrium.
This is done in terms of
eigenvalues of some matrix or matrix-valued function.
However, linear stability can fail to describe dynamics
in practice.
If solutions to the linearized system can undergo
large transient growth before eventual decay,
as can happen for systems
$\dot{x} = Mx$ when $M$ is
nonnormal~\cite{trefethen2005spectra}, then
the truncated nonlinear terms may become
significant and incite even greater growth,
rendering the linear stability analysis
irrelevant. See~\cite{singler2008},~\cite{green2006}
(the semiconductor laser model which we also study
here)
and~\cite{gebhardt1994} for examples where this
happens.

Throughout this paper we will focus on
autonomous, homogeneous, constant-coefficient linear systems,
which are the type of systems often encountered as
linearizations of nonlinear differential equations.
Work has already been done on pseudospectral upper
and lower bounds on transient dynamics for first-order
ODEs of this type
(see~\cite{trefethen2005spectra}), and those results
inspire the bounds derived here.
Additionally, upper bounds derived using Lyapunov norms
appear in~\cite{hinrichsen2006mathematical}, and a study
of these and more upper bounds, some elementary and
some requiring specific assumptions, is contained
in~\cite{plischke2005}.
As for delay differential equations (DDEs),
an upper bound has been derived based on
Lyapunov-Krasovskii
functionals applied to an operator mapping one
solution segment to the next~\cite{plischke2005};
an approximate pseudospectral lower bound is obtained
in~\cite{green2006} by discretization of the
associated infinitesimal generator as
in~\cite{bellen2000} to reduce to the ODE case;
and in~\cite{lehman1992transient} changes in the
time-average of a solution under changes to the
model are used to infer effects on transient
behavior. As far as the authors are aware, there
has been no work extending the pseudospectral
bounds in~\cite{trefethen2005spectra} to equations
beyond first-order ODEs. Our extension consists
of replacing the resolvent of a matrix, which
plays the key role in the first-order ODE results,
with the generalized resolvent of a matrix-valued
function which naturally appears in the same way.
This idea is straightforward, but a useful
implementation depends on the details of the problem
at hand. Therefore, rather than state a general
result at the cost of introducing an ungainly
and narrowly applicable set of assumptions, we
apply the main principle
to higher order ODEs and somewhat less directly
to DDEs with constant delay. We hope to motivate
the use of this idea in various other situations,
in which the necessary assumptions will be taken
into account as needed.


The rest of this paper is organized as follows.
In Section 2 we reproduce the transient growth bounds
for first-order ODEs and give the necessary background
on nonlinear matrix-valued functions and pseudospectra.
In Section 3 we make the direct extension to
higher-order ODEs and introduce our main example, a model
of a semiconductor laser with phase-conjugate feedback.
Section 4 contains our main theorem, an upper bound
for transient growth for DDEs, and its application
to a discretized partial DDE and to the laser example.
In the penultimate section, we give a practical lower
bound on worst-case transient growth and show its
effectiveness on both our examples. Finally, we
conclude in Section 6.

\section{Preliminaries}
\label{sec-transbnds-prelim}


The essential ingredient used in~\cite{trefethen2005spectra}
to derive bounds on transient growth for ODEs $\dot{x} = Mx$
is the contour integral relationship between
the solution propagator $e^{tM}$ and the matrix resolvent
$(zI-M)^{-1}$. To
paraphrase~\cite[Theorem 15.1]{trefethen2005spectra},
\begin{equation}\label{eq-resM}
 (zI-M)^{-1} = \int_0^\infty e^{-zt}e^{tM}\,dt
\end{equation}
for $\Re z$ sufficiently large, and
\begin{equation}\label{eq-etM}
 e^{tM} = \frac{1}{2\pi i}
         \int_{\Gamma} e^{zt} (zI-M)^{-1}\,dz
\end{equation}
where $\Gamma$ is a contour enclosing the eigenvalues of $M$.
The first equation means that
$\lbrace\mathcal{L}x\rbrace(z) = (zI-M)^{-1}x_0$
for any solution with initial condition $x(0) = x_0$ and
$\Re z$ sufficiently large, $\mathcal{L}$ being the
Laplace transform operator. The second equation is an
inverse Laplace transform. By exploiting the Mellin
inversion theorem~\cite{weinberger2012pde}, we will
be able to use similar integral equations to achieve
the desired bounds for more general problems. But first we will collect the
necessary terms, state the theorems from~\cite{trefethen2005spectra}
we wish to extend, and give some background on nonlinear
matrix-valued functions and pseudospectra.

We start by recalling the definition of the spectra and
pseudospectra
of matrices. The spectrum $\Lambda(M)$ of a matrix $M$
is the set of its eigenvalues.
The $\ve$-pseudospectrum,
denoted by $\Lambda_\ve(M)$, is
the union of the spectra of all matrices $M + E$, where
$\|E\| \le \ve$. An equivalent definition which gives
a different intuition is
$\Lambda_\ve(M) = \lbrace z \in \mathbb{C} :
                  \|(zI-M)^{-1}\| > \ve^{-1}
\rbrace$~\cite{trefethen2005spectra}.
The spectral abscissa $\alpha(M)$ of
$M$ is the largest real part
among any of its eigenvalues, and the pseudospectral
abscissa is correspondingly defined as $\alpha_\ve(M)
= \max \Re \Lambda_\ve(M)$. The spectral abscissa of $M$
determines asymptotic growth,
and the next theorem shows the role of the
pseudospectral
abscissa in transient growth. Its usefulness
is most apparent when $\alpha(M) < 0$.
\begin{theorem}[based on~\cite{trefethen2005spectra}, Theorem 15.2]
\label{thm-trefupper}
If $M$ is a matrix and $L_\ve$ is
the arc length of the boundary of $\Lambda_\ve(M)$
(or the convex hull of $\Lambda_\ve(M)$)
for some $\ve > 0$, then
\begin{equation}\label{eq-trefupper}
\|e^{tM}\| \le \frac{L_\ve e^{t\alpha_\ve(M)}}{2\pi\ve}
\qquad \forall t\ge 0.
\end{equation}
\end{theorem}
\begin{proof}
Let $\Gamma$ be the boundary of $\Lambda_\ve(M)$
(or its convex hull).
Since $\Lambda_\ve(M)$ contains the spectrum of $M$
for every $\ve > 0$, $\Gamma$ contains the spectrum
of $M$. Therefore we can
use the representation~(\ref{eq-etM}) for $e^{tM}$.
On $\Gamma$, $|e^{zt}| \le e^{t\alpha_\ve(M)}$ and
$\|(zI - M)^{-1}\| \le \ve^{-1}$. Taking norms
in~(\ref{eq-etM}), we then have
\[
 \|e^{tM}\| \le \frac{1}{2\pi} e^{t\alpha_\ve(M)}
  \ve^{-1} \int_{\Gamma} |dz|,
\]
and the theorem follows by observing that
$L_\ve = \int_{\Gamma} |dz|$.
\end{proof}

In addition to upper bounds, pseudospectra also give
lower bounds on the maximum achieved by $\|\exp(Mt)\|$,
as in this theorem paraphrased from~\cite[Theorem 15.5]{
trefethen2005spectra}:
\begin{theorem}\label{thm-treflower}
Let $M$ be a matrix and let $\omega \in \bbR$ be fixed.
Then $\alpha_\ve(M)$ is finite for
each $\ve > 0$ and
\[
 \sup_{t \ge 0} \|e^{-\omega t}\exp(tM)\| \ge \frac{\alpha_\ve(M)-\omega}{\ve}
 \qquad\forall\ve > 0.
\]
\end{theorem}
\begin{proof}
Letting $\ve > 0$ be arbitrary,
$\alpha_\ve(M)$ is finite because
\[
 \|(zI-M)^{-1}\| =
 |z|^{-1}\,\|\sum_{n=0}^\infty \left(z^{-1}M\right)^n \|
 \le \frac{|z|^{-1}}{1 - |z|^{-1}\,\|M\|}
\]
is less than $\ve^{-1}$ for $|z|$ sufficiently large.
Thus, the desired bound is trivially
satisfied for $\omega \ge \alpha_\ve(M)$. Therefore we assume
that $\omega < \alpha_\ve(M)$.

Now, let $z \in \Lambda_\ve(M)$ satisfy
$\Re(z) > \alpha(M)$ so that~(\ref{eq-resM})
holds, and further suppose that $\Re(z) > \omega$.
Then $\|(zI-M)^{-1}\| \ge \ve^{-1}$ by definition of
$\Lambda_\ve(M)$, and by~(\ref{eq-resM})
the hypothesis
$\|e^{Mt}\| \le C e^{\omega t}$ for all $t \ge 0$ implies that
$\|(zI - M)^{-1}\| \le \frac{C}{\Re(z) - \omega}$. This in turn
implies that $\Re(z) \le C\ve + \omega$. Since $z$ may be chosen
such that $\Re(z)$ is arbitrarily close to $\alpha_\ve(M)$, it follows
that $\alpha_\ve(M) \le C\ve + \omega$.


 By the contrapositive, if $\alpha_\ve(M) > C\ve + \omega$, then
 \[
 \sup_{t \ge 0}\|e^{-\omega t}e^{Mt}\| \ge \sup \lbrace C : \frac{
 \alpha_\ve(M)- \omega}{\ve} > C \rbrace = \frac{\alpha_\ve(M)-\omega}{\ve}.
 \]
\end{proof}
\begin{remark}
If $\omega = 0$ and $\alpha(M) < 0$, then
the essence of the theorem is that there is some
unit initial condition $x_0$ such that the solution to
$\dot{x} = Mx$, $x(0) = x_0$ satisfies $\|x(t_0)\| \ge
\sup_{\ve > 0} \frac{\alpha_\ve(M)}{\ve}$ at some finite
time $t_0$ before eventually decaying to zero.
One can think of such a solution as a long-lived
``pseudo-mode'' associated with pseudo-eigenvalues
in the right half plane. If the $\ve$-pseudospectrum
extends far into the right half plane for some small
$\ve$, then there must be some solution that exhibits
large transient growth.
\end{remark}

To bound transient growth for the higher-order ODE and DDE cases,
a generalized resolvent $T(z)^{-1}$,
with $T : \Omega \rightarrow \bbC^{n\times n}$
an analytic, nonlinear matrix-valued function, will
play the role that the resolvent $(zI-M)^{-1}$ did above.
We say that $\lambda$ is an
eigenvalue of $T$ if $T(\lambda)$ is singular, and
let $\alpha(T)$ (the pseudospectral abscissa of $T$)
represent the largest real part of any eigenvalue of $T$.
A very general
definition of the $\ve$-pseudospectrum of a matrix-valued
function, and
the one we will use, is
\[
\Lambda_\ve(T) = \lbrace z \in \Omega : \|T(z)^{-1}\| > \ve^{-1}
                \rbrace
\]
(see~\cite{bindelhood} for the motivation behind
this particular definition and references to alternative definitions).
Now we can define the pseudospectral abscissa of $T$ as
\[
 \alpha_\ve(T) = \sup_{z \in \Lambda_\ve(T)} \Re z.
\]
The following proposition is immediate.
\begin{proposition}
If $\|T(z)^{-1}\| \rightarrow 0$ uniformly as
$\Re z \rightarrow \infty$, then $\alpha_\ve(T) < \infty$
for all $\ve > 0$.
\end{proposition}



\section{Upper bounds for higher-order ODEs}
\label{sec-upperHODEs}

In this section we treat equations of the form
\begin{equation}\label{eq-hode}
 y^{(n)} = \sum_{j=0}^{n-1} A_j y^{(j)}
\end{equation}
with initial conditions $y^{(j)}(0) = y_0^{(j)}$,
$j = 0, ..., n-1$, and
where each $A_j \in \mathbb{C}^{k \times k}$.

We can solve~(\ref{eq-hode}) by writing it in
first-order form, e.g., $\dot{x} = Mx$,
$x = [y, \dot{y}, ..., y^{(n-1)}]^T$, where
\begin{equation}
 M = \begin{bmatrix}
     0   & I   & 0   & \hdots & 0 \\
     0   & 0   & I   & \hdots & 0 \\
     0   & 0   & 0   & \ddots & 0 \\
     0   & 0   & 0   & \hdots & I \\
     A_0 & A_1 & A_2 & \hdots & A_{n-1}
     \end{bmatrix}.
\end{equation}
Then Theorem~\ref{thm-trefupper} can be applied to the
solutions $x(t) = e^{Mt}x(0)$ in order to bound $y(t)$,
since $\|y(t)\| \le \|x(t)\|$.
But since the maximum
reached by $\|x(t)\|$ could be much larger than the maximum of
$\|y(t)\|$, one can hardly expect to obtain a tight bound
for $\|y(t)\|$ with this process.
With the next
theorem, we can bound $y(t)$ directly.
\begin{theorem}
\label{thm-schurUB}
 Let the equations in $y$ and $x$ be as above, with
 $M$ partitioned as
 \[
  M = \begin{bmatrix}0 & B\\ C & D\end{bmatrix},
  \qquad D\text{ square}.
 \]
 Assume $y_0^{(j)} = 0$ for $j = 1,...,n-1$.
 Then $y(t) = \Psi(t) y(0)$ with
 \[
  \|\Psi(t)\| \le \frac{L_\ve e^{\alpha_\ve(T)t}}{2\pi\ve}\qquad \forall\ve > 0,
 \]
 where $T(z) = zI - B(zI-D)^{-1}C$ and $L_\ve$ is the arc length of
 the $\ve$-pseudospectrum of $T$. The bound is finite for every $\ve > 0$.
\end{theorem}
\begin{proof}
 Let $E_1$ represent the first
 $k$ columns of the $nk\times nk$ identity. Then
 the initial condition in $x$ is $x(0) = E_1 y(0)$, so that
 $x(t) = e^{Mt}E_1 y(0)$ and hence $y(t) = E_1^T e^{Mt} E_1 y(0)$.
 Therefore we define $\Psi(t):= E_1^T e^{Mt} E_1$.
 From the integral representation~(\ref{eq-etM}) for $e^{Mt}$, we have
 \[
  \Psi(t) = \frac{1}{2\pi i}
            \int_\Gamma
            \underbrace{E_1^T (zI-M)^{-1} E_1}_{T(z)^{-1}}
  e^{zt}\,dz.
 \]

 As in the proof of Theorem~\ref{thm-treflower},
for $|z|$ large enough
 $\|T(z)^{-1}\| \le \ve^{-1}$.
 Therefore
 $\Lambda_\ve(T)$ is bounded for every $\ve > 0$.
 Therefore $L_\ve$ and $\alpha_\ve(T)$ are both finite
 and the result follows as in
 Theorem~\ref{thm-trefupper}.
\end{proof}
\begin{remark}
 Bounds for similar objects can be found in~\cite{plischke2005}
 in the section ``Kreiss Matrix and Hille-Yosida Generation Theorems,''
 where structured $(M,\beta)$-stability is considered.
\end{remark}

 The assumption $y^{(j)}(0) = 0$ for $j = 1,...,n-1$
 was not essential, as the following corollary
 shows.
 \begin{corollary}
 If $y$ satisfies $y^{(n)} = \sum_{j=0}^{n-1} A_j y^{(j)}$,
 with initial condition
 $y^{(j)}(0) = y_0^{(j)}$ for $j = 0,1,...,n-1$, then
 \[
  \|y(t)\| \le \sum_{j=0}^{n-1}
               \frac{L_\ve^{(j)} e^{\alpha_\ve(T_j)t}}
                    {2\pi \ve} \|y_0^{(j)}\|,
  \qquad \forall \ve > 0,
 \]
 where $L_\ve^{(j)}$ is the arclength of the boundary of
 $\Lambda_\ve(T_j)$, $T_j(z)^{-1} = E_1^T (zI-M)^{-1} E_{j+1}$,
 and $E_{j+1}$ is the $j+1$-th block column of the $nk \times nk$ identity
 partitioned into $n$ block columns.
\end{corollary}
\begin{proof}
From $x(0) = \sum_{j=0}^{n-1} E_{j+1}y_0^{(j)}$, we can
use~(\ref{eq-etM}) to write
 \[
  y(t) = \sum_{j=0}^{n-1} \frac{1}{2\pi i}
                        \int_\Gamma T_j(z)^{-1} e^{zt}\,dz \cdot
                        y^{(j)}(0),
         \qquad
         T_j(z)^{-1} = E_1^T (zI-M)^{-1} E_{j+1}
 \]
 and apply the theorem to each summand.
\end{proof}
\begin{remark}
 We arrive at expressions for each $T_j(z)$
 by taking the Laplace transform of the original
 equation~(\ref{eq-hode}) and expressing $y(t)$ in
 terms of the inverse Laplace transform.
 First, using standard facts about
 the Laplace transform,
 \[
  s^n Y(s) - \sum_{j=0}^{n-1} s^{n-1-j} y_0^{(j)}
  =
  \sum_{k=0}^{n-1} A_k
  \left(s^k Y(s) - \sum_{j=0}^{k-1} s^{k-1-j} y_0^{(j)}\right),
  \qquad Y = \mathcal{L}y.
 \]
 Rearranging,
 \[
  \underbrace{
  \left(s^nI - \sum_{k=0}^{n-1} s^kA_k\right)
             }_{P(s)}  Y(s)
  =
  \sum_{j=0}^{n-1}
  \underbrace{
   \left( s^{n-1-j}I - \sum_{k = j+1}^{n-1} A_k s^{k-1-j}\right)
             }_{X_j(s)}   y_0^{(j)}.
 \]
 Then we recover
 \[
  y(t) = \sum_{j=0}^{n-1} \frac{1}{2\pi i}
         \int_{\Gamma_j} P(z)^{-1}X_j(z)\, e^{zt}\,dz
         \cdot y_0^{(j)}
 \]
 from which we see that $T_j(z)^{-1} = P(z)^{-1}X_j(z)$. Therefore
$T_j(z) = X_j(z)^{-1}P(z)$.
\end{remark}

The last result of this section is the higher-order
difference equation version of the last corollary, and
is a direct extension
of~\cite[Theorem 16.2]{trefethen2005spectra}.
\begin{corollary}\label{thm-diffeqUB}
Suppose $(y_n)$ satisfies the difference equation
$y_{n+1} = \sum_{j=0}^N A_j y_{n-j}$ with initial conditions
$y_0,y_{-1},...,y_{-N}$ given. Then
\[
 \|y_n\| \le \sum_{j=0}^N \frac{L_\ve^{(j)}
             \rho_\ve(T_j)^n}{2\pi\ve} \|y_{0-j}\|
         < \infty
 \qquad
 \forall \ve > 0
\]
where $T_j(z)^{-1} = E_1^T (zI-M)^{-1} E_{j+1}$,
$L_\ve^{(j)}$ is the arclength of the boundary of $\Lambda_\ve(T_j)$,
and $\rho_\ve(T_j) = \max \lbrace |z| : z \in \Lambda_\ve(T_j) \rbrace$ is
the pseudospectral radius.
\end{corollary}
\begin{proof}
Putting $x_n = [y_{n-0},y_{n-1},...,y_{n-N}]^T$, we have $x_n = M^n x_0$.
Applying the inverse Z-transform to $z(zI-M)^{-1}$ we
obtain $M^n = \frac{1}{2\pi i}\int_\Gamma z^n(zI-M)^{-1}\,dz$ for $\Gamma$
a contour enclosing the spectrum of $M$. The quantity of interest may
then be expressed as $y_n = E_1^Tx_n = \sum_{j=0}^N E_1^T M^n E_{j+1}y_{0-j}$.
Since $E_1^T M^n E_{j+1} = \frac{1}{2\pi i}\int_{\Gamma_j} z^n T_j(z)^{-1}\,dz$,
the bound for this term follows by taking $\Gamma_j$ equal to the
$\ve$-pseudospectrum of $T_j$. These bounds are finite for any $\ve$ since
$\|(zI-M)^{-1}\| \ge \|T_j(z)^{-1}\|$ for all $z$ implies that
$\Lambda_\ve(T) \subset \Lambda_\ve(M)$, and we know the latter to be
finite.
\end{proof}

Our first example demonstrates the improvement
in bounding the solution to~(\ref{eq-hode}) directly
versus bounding the solution to the first-order form $\dot{x} = Mx$
while simultaneously
motivating the need for the bound in the next section.
\begin{example} \rm
 The model for a semiconductor laser with phase-conjugate feedback
 studied in~\cite{green2006} has an equilibrium at
\[
 (E_x,E_y,N) = (+1.8458171368652383, -0.2415616277234652, +7.6430064479131916)
\]
 after scaling, and linearizing about
this equilibrium yields the DDE
$\dot{y}(t) = A y(t) + B y(t-1)$ where
\footnote{
The equilibrium and linearized system computed here differ slightly
from those in~\cite{green2006}. It appears there were two
 typographical errors and a lack of precision in one or more
 of the given parameters, which led to the parameters,
 linearization, and equilibrium stated
 in~\cite{green2006} being mutually inconsistent. The authors
 have been contacted, and we have attempted to reproduce their
 linearization and equilibrium closely by making the following
 adjustments.
We have put $I = 0.0651 A$ to match
the value
in~\cite{gray1994laser},~\cite{green2002},
and~\cite{krauskopf1998}
cited in~\cite{green2006} as the sources for
the parameters. We have set $\kappa = 4.2\times 10^8 s^{-1}$ so that the
coefficient matrix $B$ for the delay term is the same as
in~\cite{green2006}. Lastly, we have put
$N_{sol} = N_0 + 1/(G_N \tau_p)$ as prescribed
in~\cite{green2002} (which is approximately
the value of $N_{sol}$ stated in~\cite{green2006}).
The equilibrium we have listed above was computed
using Newton's method after the parameter
corrections were implemented, with the equilibrium
stated in~\cite{green2006} as an initial guess.
}
\begin{align*}
 A &= \begin{bmatrix}
  -8.4983\times 10^{-1} &  1.4786\times 10^{-1} &  4.4381\times 10^1 \\
   3.7540\times 10^{-3} & -2.8049\times 10^{-1} & -2.2922\times 10^2 \\
  -1.7537\times 10^{-1} &  2.2951\times 10^{-2} & -3.6079\times 10^{-1}
       \end{bmatrix}, \\
 B &= \begin{bmatrix}
   2.8000\times 10^{-1} &           0 &          0 \\
            0 & -2.8000\times 10^{-1} &          0 \\
            0 &           0 &          0
       \end{bmatrix}.
\end{align*}

Discretizing with $N+1$ points on each unit segment, we can use the
forward Euler approximation to obtain the higher-order difference
equation approximation $y_{j+1} = (I + hA)y_j + hBy_{j-N}$, $h = 1/N$.
Initial conditions $y_{0},y_{0-1},...,y_{0-N}$ come from sampling
the initial condition for the original equation on $[-1,0]$.
We then apply Corollary~\ref{thm-diffeqUB}
with $A_N = hB$, $A_0 = I+hA$, and $A_j = 0$ otherwise.
A companion linearization gives
\[
 \underbrace{ \begin{bmatrix}y_{j+1}\\ \vdots \\ y_{j-N+1}\end{bmatrix}
            }_{x_{n+1}}
 =
 \underbrace{ \begin{bmatrix}I+hA & 0      & \hdots & hB \\
                              I   &        &        & 0    \\
                                  & \ddots &        & \vdots \\
                                  &        &   I    & 0\end{bmatrix}
            }_M
 \underbrace{ \begin{bmatrix}y_j \\ \vdots \\ y_{j-N}\end{bmatrix}
            }_{x_n}.
\]
Here we choose the initial condition $y(t) = 0.0015\times (Ex,Ey,N)^T$ on $[-1,0]$ since a similar initial condition in~\cite{green2006} corresponds to a decaying solution with nontrivial transient growth.

As in the continuous case, with a little manipulation we find that
$T_j(z)^{-1} = P(z)^{-1}X_j(z)$, where $P(z) = z^{N+1}I - (I+hA)z^N - hB$,
$X_0(z) = z^N I$, and $X_j(z) = hBz^{j-1}$ for $j \ge 1$. (In general,
$P(z) = z^{N+1}I - \sum_{j=0}^N A_j z^{N-j}$ and $X_j(z) =
\sum_{k=j}^N A_k z^{N+j-1-k}$ for $j \ge 1$.) Notice that $B$ is
singular and therefore the inverse of $T_j(z)^{-1}$ does not exist for $j \ge 1$.
However, we can
still do $\|T_j(z)^{-1}\| \le \|P(z)^{-1}\| h\|B\|\,|z^{j-1}|$
and obtain bounds by taking
$\Gamma_j = \partial\lbrace z : \|P(z)^{-1}\| h\|B\|\,|z|^{j-1} = \ve^{-1}\rbrace$
for $j \ge 1$. That set is easy to compute in terms of pseudospectra
for $P(z)$. In Figure~\ref{fig-DDEdisc}, we show an upper bound on
$\|y_n\|_2$ from using Corollary~\ref{thm-diffeqUB}, an upper upper bound
on $\|x_n\|_2$ using Theorem~\ref{thm-trefupper}, the 2-norm of the solution
$y_n$ and the 2-norm of the solution $y(t)$ to the continuous DDE.
Notice that as the mesh becomes finer, $y_n$ becomes a better approximation
to $y(t)$ but the upper bound on $\|y_n\|_2$ becomes much more generous.
This is because the spectral radius of $M$ increases with mesh size.
This in turn suggests that a bound which comes directly from the continuous DDE
itself may be more straightforward and effective.

\begin{figure}
\includegraphics[width=0.32\linewidth]{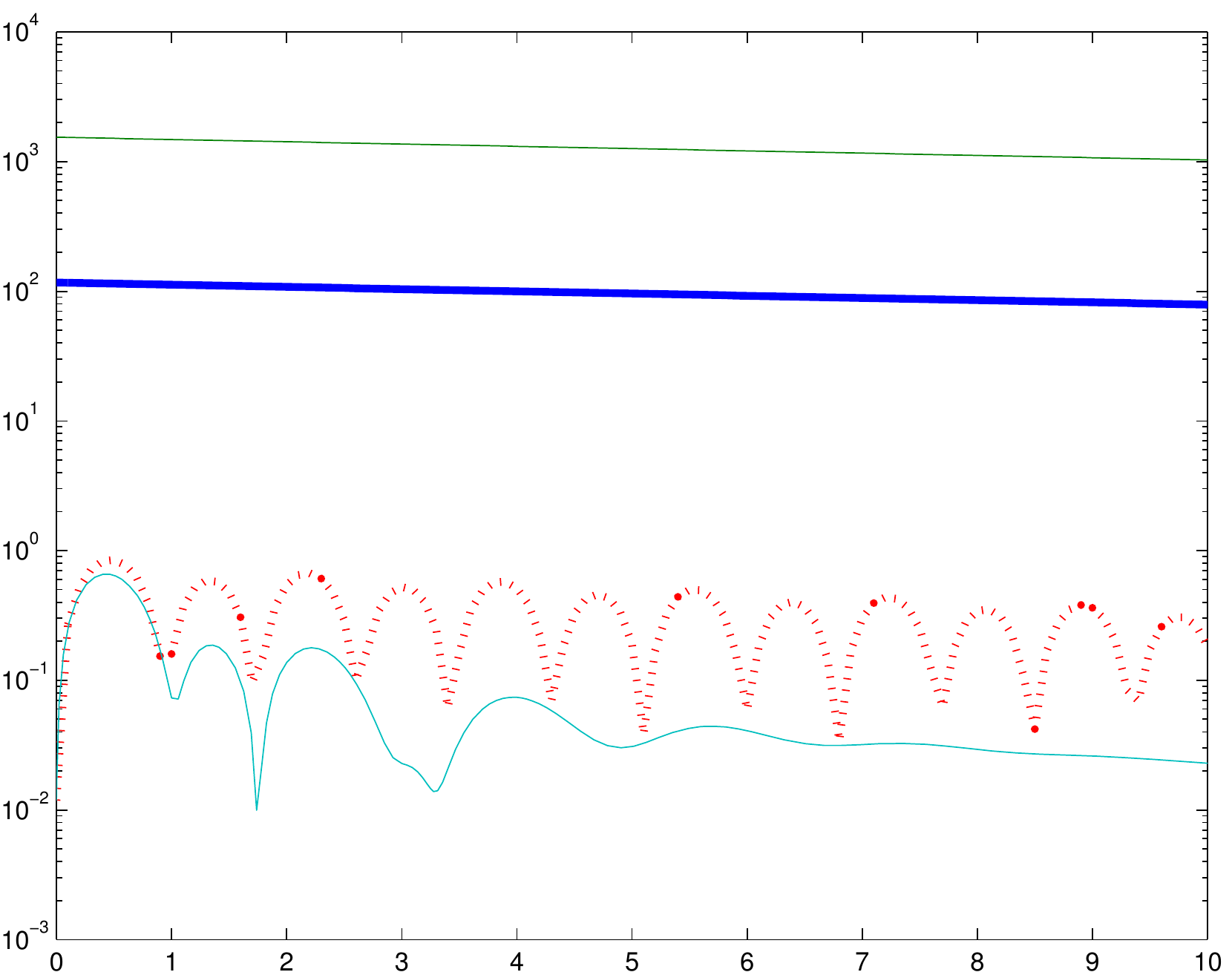} \hfill
\includegraphics[width=0.32\linewidth]{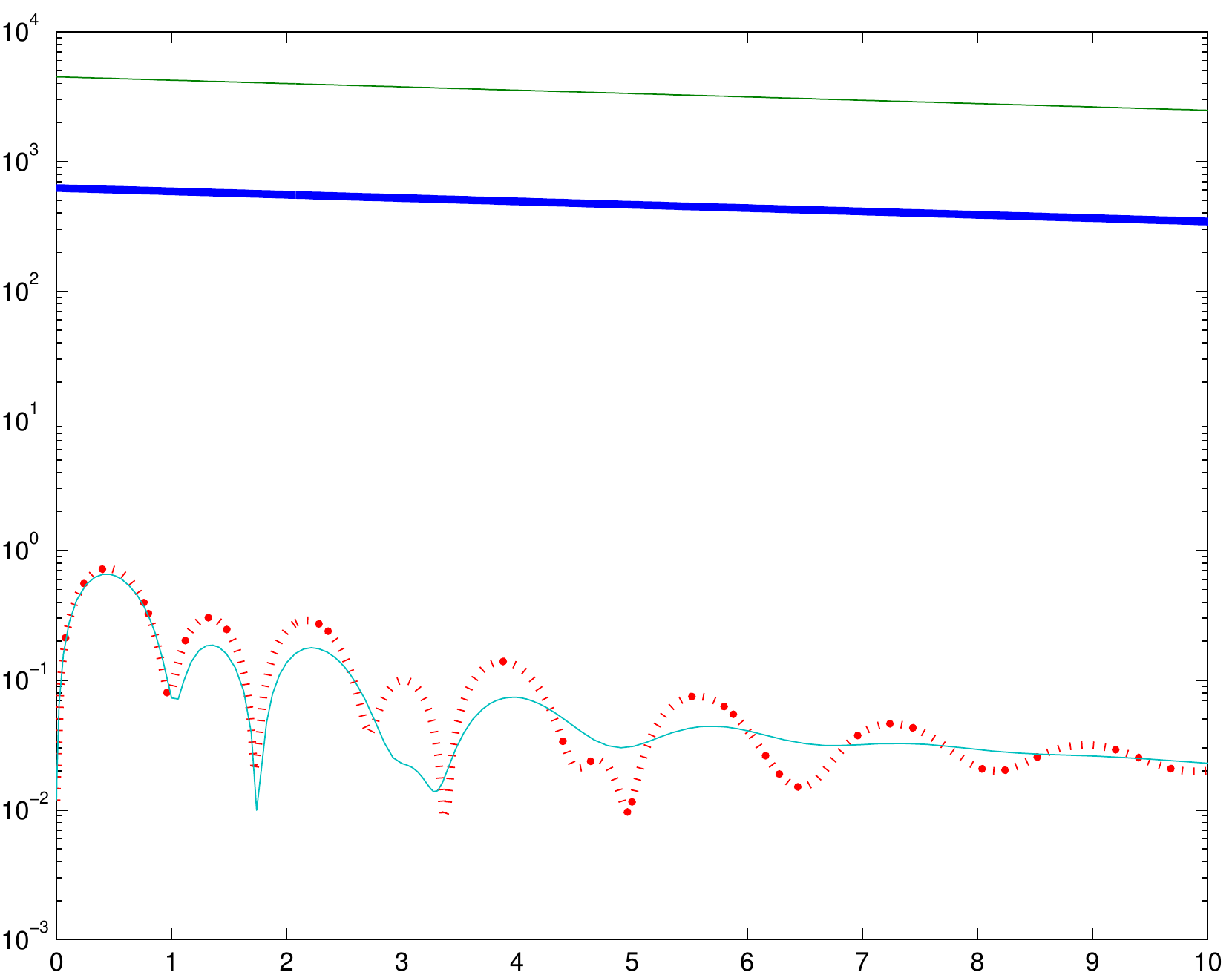} \hfill
\includegraphics[width=0.32\linewidth]{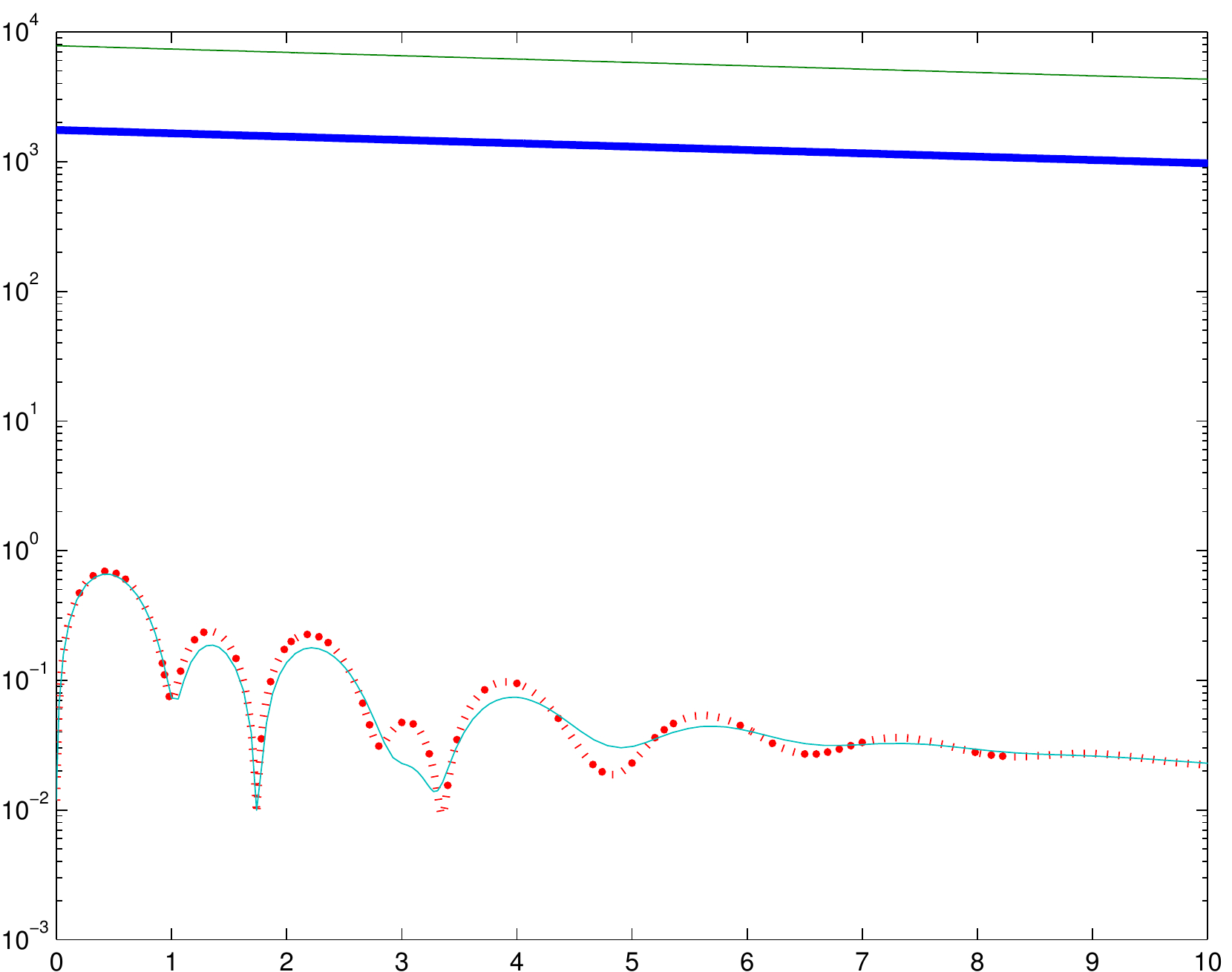}
\caption{
Upper bound from Corollary~\ref{thm-diffeqUB} (thick, solid),
upper bound from Theorem~\ref{thm-trefupper},
2-norm of solution $y(t)$ to continuous DDE computed
with the Matlab dde23 routine, and 2-norm of solution
$y_n$ to discretized equation (thick, dotted) for
$N = 10$ (left), $N =25$ (center), $N = 50$ (right).
}\label{fig-DDEdisc}
\end{figure}
\end{example}

\section{Upper bounds for delay differential equations}
\label{sec-upperDDEs}

Now we turn to transient bounds for DDEs
\begin{equation} \label{eq-dde}
 \dot{u}(t) = Au(t) + Bu(t-\tau)
 \qquad A,B \in \bbC^{n\times n}
\end{equation}
with a single delay $\tau > 0$
and with $\alpha(A)$ and $\alpha(T)$ both negative, where
$T(z) = zI-A-Be^{-\tau z}$ is the characteristic
equation~\cite{MichielsNiculescu}.
Although
we treat only a single delay here, a direct
extension to multiple constant delays is
straightforward.

The characteristic equation $T(z) = zI-A-Be^{-\tau z}$
generally has infinitely
many eigenvalues, as is often the case with
nonlinear matrix-valued
functions. Therefore, unlike in the previous section,
bounds on transient behavior will depend on integrals
whose integration path is unbounded.
So, if we expect a bound for a DDE
to be useful in practice,
then we expect it to
require more preparation (such as locating eigenvalues so
as to find an admissible integration path)
and look more complicated (since the integrand's behavior at
infinity will need to be analyzed) than a bound for an ODE.

Let $\Psi$ be the fundamental solution for the DDE,
that is, the
solution whose initial conditions are zero on $[-\tau,0)$
and the matrix identity $I$ at $t = 0$.
Following the treatment in Chapter 1 of~\cite{hale1993},
we first bound $\Psi$
by invoking the
Mellin inversion theorem~\cite{weinberger2012pde}
and then splitting the characteristic
equation into its linear and
nonlinear parts. We show that
the integration can be taken over a curve more convenient than
the usual vertical one, and finally compute upper bounds
using elementary means.

\begin{lemma}\label{thm-neumann-bound}
 If $X,Y \in \bbC^{n\times n}$ are two matrices, and
 if $\|X^{-1}\|\,\|Y\| < 1$, then
 \[
  \|(X - Y)^{-1}\| \le
  \frac{1}{\|X^{-1}\|^{-1} - \|Y\|}.
 \]
\end{lemma}
\begin{proof}
 We write $(X - Y)^{-1} = (I - X^{-1}Y)^{-1} X^{-1}$. By
 hypothesis, the Neumann series
 $\sum_{j=0}^\infty (X^{-1}Y)^j$ for $(I - X^{-1}Y)^{-1}$
 converges and is bounded by $(1 - \|X^{-1}\|\,\|Y\|)^{-1}$.
 Therefore $\|(X - Y)^{-1}\|
 \le \|X^{-1}\|\,(1 - \|X^{-1}\|\,\|Y\|)^{-1}$ and the desired
 result follows.
\end{proof}
Lemma~\ref{thm-neumann-bound} can be used to bound
$\|T(z)^{-1}\| = \|(zI - A - Be^{-\tau z})^{-1}\|$
in various ways, depending on the choice of norm and the properties
of the matrices $A$ and $B$. For simplicity, in the following
lemma we use $\|\cdot\| = \|\cdot\|_2$ and assume $A$ is
Hermitian.
Generalizations are straightforward. For instance,
the laser example analyzed in this section
does not have $A$ Hermitian and we show
how to apply our results to that case.
\begin{lemma}
\label{thm-DDEgamma}
 Let $T(z) = zI - A - Be^{-\tau z}$ with $A = VDV^\ast$ Hermitian.
 If $y_0$ is a given positive number, and $\eta$ is chosen
 so that
\[
 1 < \eta y_0 < \min\bigg\lbrace \frac{y_0}{\|B\|_2},
                  e^{-\alpha(T)\tau}, e^{-\alpha(A)\tau} \bigg\rbrace,
\]
 then
 \[
  \Gamma = \underbrace{
           \lbrace x(y) + iy: |y| > y_0 \rbrace
                      }_{\Gamma_\infty}
           \cup
           \underbrace{
           \lbrace x_0 + iy: |y| \le y_0, x_0 = x(y_0) \rbrace
                      }_{\Gamma_0},
  \qquad
  x(y) = -\frac{1}{\tau}\log \left(|y| \eta\right)
 \]
 is to the right of both $\Lambda(T)$ and $\sigma(A)$ but lies
 entirely in the left half-plane.
\end{lemma}
\begin{proof}
$\Gamma$ is certainly to
the right of $\sigma(A)$, since all eigenvalues of $A$ are real
and the condition $\eta y_0 < e^{-\alpha(A)\tau}$ guarantees $x(y_0) >
\alpha(A)$. The eigenvalues of $T$ are also to the left of
$\Gamma_0$ by the condition $\eta y_0 < e^{-\alpha(T) \tau}$.
As for $\Gamma_\infty$,
$\|(zI - A)^{-1}\|_2^{-1} = \sigma_{\min}(zI - D) \ge |y|$ because
the eigenvalues of $A$ are real.
Hence, if $z \in \Gamma_\infty$, then
$\|(zI - A)^{-1}\|_2^{-1} \ge |y| > \|B\|_2 \eta |y|$ by
the hypothesis $\eta < 1/\|B\|_2$. Therefore,
Lemma~\ref{thm-neumann-bound} applies with $X = zI - A$ and
$Y = Be^{-\tau z}$, so
$\|T(z)^{-1}\| \le \left(|y| - \|B\|_2 \eta |y| \right)^{-1}
< \infty$ on $\Gamma_\infty$.
Since decreasing $\eta$ to
zero moves $\Gamma_\infty$ infinitely to the right, and decreasing $\eta$
does not violate the assumption guaranteeing nonsingularity of $T$ on
$\Gamma_\infty$, it follows that $T$ is nonsingular on and at all points
to the right of $\Gamma_\infty$. Therefore $\Gamma$ is to the right of
$\Lambda(T)$. The condition $1 < \eta y_0$ assures that $\Gamma_0$ is
in the left half-plane, and therefore so is $\Gamma$.
\end{proof}

By our assumption that all eigenvalues of $T$ are in the left
half plane, all
solutions of~(\ref{eq-dde}) are exponentially
stable~\cite[Proposition 1.6]{MichielsNiculescu} and
hence of exponential order. Therefore we can
take the Laplace transform of~(\ref{eq-dde}) to
obtain
\[
 (zI - A - Be^{-\tau z})U(z) = u(0)
 + Be^{-\tau z} \int_{-\tau}^0 e^{-zt}u(t)\,dt,
 \quad U = \calL u.
\]
Since the fundamental solution
satisfies $\Psi(t) = 0$ on $[-\tau,0)$ and
$\Psi(0) = I$, it follows that
$(\mathcal{L}\Psi)(z) = T(z)^{-1}$.
Then
we can use the Mellin inversion
theorem~\cite{weinberger2012pde} to write
\begin{equation}\label{eq-mellin}
 \Psi(t) = \frac{1}{2\pi i} \int_{\gamma + i\mathbb{R}} T(z)^{-1} e^{zt}\,dz
\end{equation}
for any $\gamma > \alpha(T)$. The next lemma
shows that we can integrate over the contour
$\Gamma$ from Lemma~\ref{thm-DDEgamma} rather than
$\gamma + i\bbR$ in~(\ref{eq-mellin}).
\begin{lemma}
\label{thm-DDEcauchy}
 For $\Gamma$ as in Lemma~\ref{thm-DDEgamma} and $\gamma$ such
 that~(\ref{eq-mellin}) holds, we have
 \[
  \int_\Gamma T(z)^{-1} e^{zt} \,dz =
  \int_{\gamma + i\mathbb{R}} T(z)^{-1} e^{zt} \, dz.
 \]
\end{lemma}
\begin{proof}
Since $T$ has no eigenvalues in the region bounded by
$\gamma + i\bbR$ and $\Gamma$, we only need to show that
the integrals
\[
 \int_{x(y)}^\gamma T(w + iy)^{-1}\, e^{(w+iy)t}\,dw
\]
go to zero as $y \rightarrow \pm\infty$. But
from Lemma~\ref{thm-neumann-bound}
we know $\|T(w+iy)^{-1}\|
\sim \frac{1}{|y|}$ on $x(y) \le w \le \gamma$
as $|y|$ becomes large, and
$|e^{(w+iy)t}| \le e^{\gamma t}$ on the integration
path which itself has arc length $\sim \log(|y|)$.
Therefore
\[
 \Bigg\|\int_{x(y)}^\gamma T(w+iy)^{-1}\, e^{(w+iy)t}\,dw\Bigg\|
\lesssim
 \frac{\log |y|}{|y|} \rightarrow 0
\]
as $|y|$ becomes large, and the lemma is proved.
\end{proof}

We now come to the main result of this section, in which
we bound transient growth of the fundamental solution.
Note that a bound on $\Psi(t)$ for $t \ge \tau$ is all
that is required, since $\Psi(t) = e^{At}$ for $0 \le t < \tau$.
Again, we use the 2-norm, but only for simplicity.
\begin{theorem}
\label{thm-DDEUB}
With the hypotheses of the previous lemmas,
the fundamental solution of $\dot{u}(t) = Au(t) + Bu(t-\tau)$
satisfies the bound
\[
 \|\Psi(t)\|_2 \le \|\exp(At)\|_2
              +  e^{x_0 t} I_0
              +  e^{x_0 t}
                 \frac{C}{t/\tau}
\]
on $t \ge \tau$,
where
\[
 I_0 = \frac{1}{2\pi}\int_{-y_0}^{y_0}
         \| T(x_0 + iy)^{-1} - R(x_0 + iy) \|_2
         dy,\quad
       R(z) = (zI-A)^{-1}
\]
and
\[
 C = \frac{\|B\|_2\eta\sqrt{(\tau y_0)^{-2}+1}}
          {\pi\left(1-\|B\|_2 \eta\right)}.
\]
\end{theorem}
\begin{proof}
Since $\Gamma$ was chosen to the right of
all eigenvalues of $A$, the splitting
\[
 \frac{1}{2\pi i} \int_\Gamma T(z)^{-1} e^{zt}\,dz =
 \frac{1}{2\pi i} \int_\Gamma R(z) e^{zt}\,dz +
 \frac{1}{2\pi i} \int_\Gamma \left[ T(z)^{-1} - R(z) \right]
 e^{zt}\, dz
\]
and subsequent evaluation of the first summand as $e^{At}$
is justified, as the Mellin inversion theorem applies
to $R(z)$ for the same reason it applies to
$T(z)^{-1}$.
With $I_0$ as defined in the theorem statement, it
only remains to give a bound on
the second integral in the sum.

From the hypothesis that $A$ is Hermitian we have
that $\|R(z)\|_2 \le |y|^{-1}$, and hence
$\|R(z) B e^{-\tau z}\|_2 \le \|B\|_2\eta$. Therefore the
assumption $\eta y_0 < y_0/\|B\|_2$ implies
$\|R(z) B e^{-\tau z}\|_2 < 1$ on $\Gamma_\infty$, so that
 $T(z)^{-1} - R(z)$ is subject to
the Neumann bound
\[
 \|T(z)^{-1} - R(z)\|_2
  \le |y|^{-1}\|B\|_2 \eta
      \left( 1 - \|B\|_2 \eta \right)^{-1}.
\]
In addition, if $z \in \Gamma_\infty$ then $|z'(y)|
\le \sqrt{ \left(\tau y_0\right)^{-2} +1 }$.
It then follows that
\[
 \bigg\| \frac{1}{2\pi i}
         \int_{\Gamma_\infty} \left[ T(z)^{-1} - R(z) \right]
                         e^{zt}\,dz
 \bigg\|_2
 \le
 e^{x_0 (t-\tau) } \frac{C}{ t/\tau }.
\]
\end{proof}
\begin{remark}\label{rm-DDEUB}
In general, if $A$ is not Hermitian we can still bound
$\|R(z)\|$ simply by splitting $A$ into
its Hermitian and skew-Hermitian parts as $A = H + S$, from which
$\|R(z)\|_2 \le (|y| - \|S\|_2)^{-1}$ if we use the 2-norm.
However $\|R(z)\|$ is bounded
must be taken into account when choosing $\Gamma_\infty$.

Also note that we could have integrated over a vertical
contour at $\Re z = x_0$, because $\|T(x_0+iy)^{-1}-R(x_0+iy)\| \lesssim
|y|^{-2}$ as $|y| \rightarrow \infty$. But then we obtain
\begin{equation}\label{eq-DDEUB-vert}
 \|\Psi(t)\|_2 \le \|\exp(At)\|_2 + e^{x_0 t}I_0 + e^{x_0 t}C,
 \qquad
 C = \frac{\|B\|_2\eta}{\pi(1-\|B\|_2 \eta)}
\end{equation}
and we have lost the $1/t$ dependence
in the third term.

Since $\|T(z)^{-1}\|$ is integrable on $\Gamma_\infty$, we also
could have obtained an upper bound in terms of the integral of
$\|T(z)^{-1}\|$ over $\Gamma_0$ and another term with $1/t$ dependence,
specifically
\begin{equation}\label{eq-DDEUB-nonsplit}
 \|\Psi(t)\|_2 \le e^{x_0 t} \tilde{I_0} + e^{x_0t}\frac{C}{t/\tau},
 \qquad
 \tilde{I_0} = \frac{1}{2\pi}\int_{\Gamma_0} \|T(z)^{-1}\|_2 \,|dz|,
\ \
 C = \frac{\sqrt{(\tau y_0)^{-2} + 1}}{\pi (1-\|B\|_2\eta)}.
\end{equation}
In this version we fail to take advantage of the closed form of
$\int R(z) e^{zt}\,dz$. However, we may be able to shift $x_0$
further to the left since we no longer need to have $\Gamma$ to
the right of the spectrum of $A$, and this will result in
faster decay.
\end{remark}

One consequence of a bound on the fundamental solution is a bound
on worst-case transient behavior of a certain class of solutions,
namely the ones equal to zero on $[-\tau, 0)$ and with an initial
``shock'' condition specified at $t = 0$.

\begin{example} \rm
\label{ex-PDDE}
 Consider the DDE
 \begin{equation}\label{eq-effex}
  \dot{v}(t) = A v(t) + B v(t-\tau), \quad \tau = 0.2
 \end{equation}
 coming from the discretization of a parabolic partial differential
delay equation (adapted from \cite[\S 5.1]{effenberger2013}),
 where $n = 10^2$ and $h = \pi/(n+1)$,
 and $A, B \in \bbC^{n,n}$ are defined by
 \[
  A = \frac{1}{h^2}
        \begin{pmatrix}
          -2 &    1   &    \\
           1 & \ddots &  1 \\
             &    1   & -2
        \end{pmatrix} + \frac{1}{2} I,
  \qquad
  B(j,j) = a_1(jh),\quad
  a_1(x) = -4.1 + x(1-e^{x-\pi}).
 \]
Rather than compute eigenvalues of $T$ to obtain $\alpha(T)$,
we can compute inclusion
regions~\cite[Theorem 3.1]{bindelhood} with the splitting
$T(z) = D(z) + E(z)$,
$D(z) = zI-\Lambda$ diagonal and $E(z) = V^{-1}BV e^{-\tau z}$, where $A = V\Lambda V^{-1}$.
The resulting inclusion regions are plotted in Figure~\ref{fig-PDDE-UB} (left). The rightmost
point of the inclusion regions is then a bound for $\alpha(T)$.
The largest eigenvalues of $A$ are also plotted, and the solid contour
is chosen as in the theorem, with $y_0 = 21.4214$ and $\eta = 0.0491366$,
so that it is to the right of both the eigenvalues of $T$ and the
eigenvalues of $A$. The dashed contour is the vertical alternate integration
path, as referred to in Remark~\ref{rm-DDEUB}. Lastly, if we
use~(\ref{eq-DDEUB-nonsplit}), we can integrate over a contour whose
vertical section is shifted to the left, depicted as the dotted line in
Figure~\ref{fig-PDDE-UB} (left).
\begin{figure}[h]
 \includegraphics[width=0.48\linewidth]{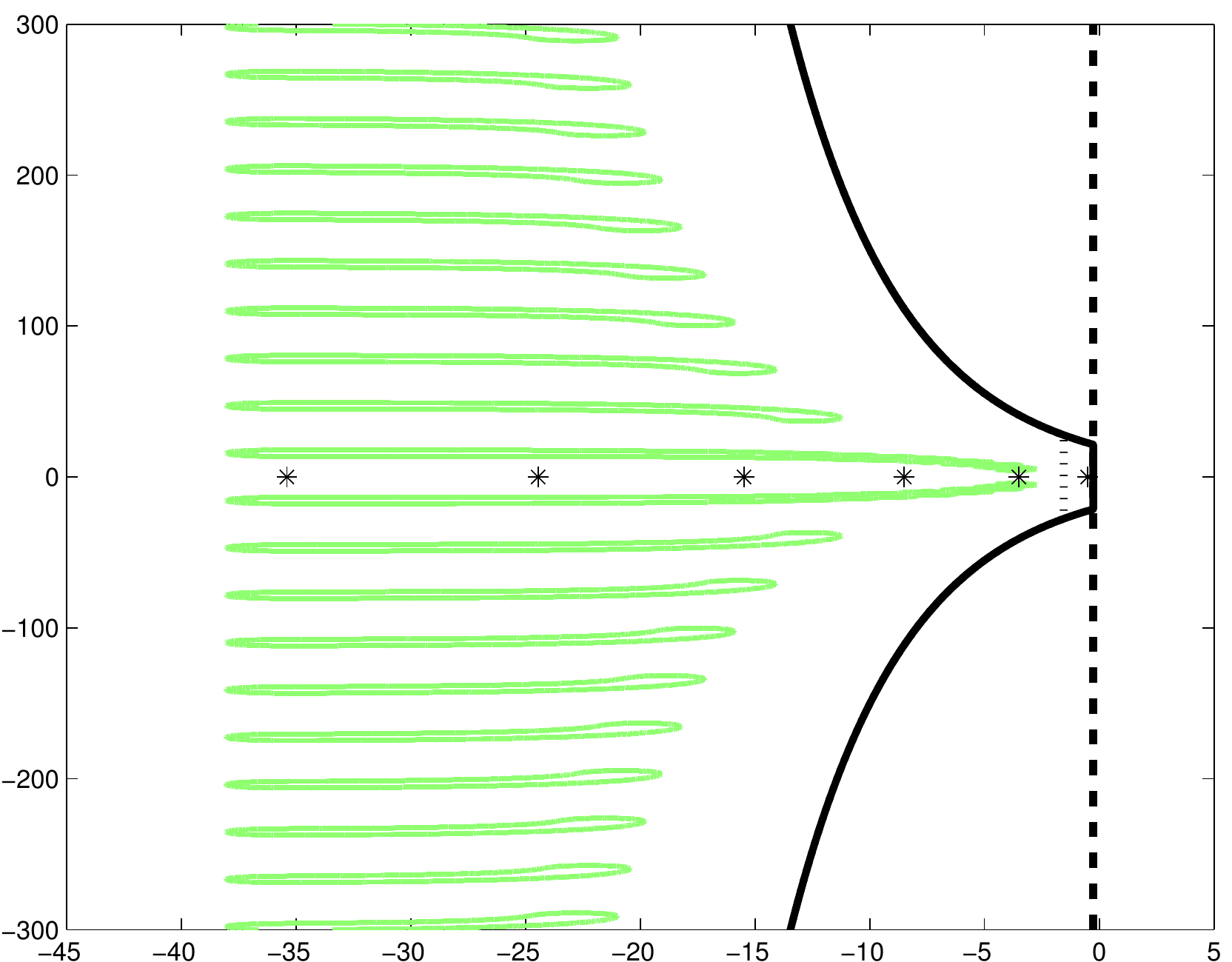}
 \hfill
 \includegraphics[width=0.48\linewidth]{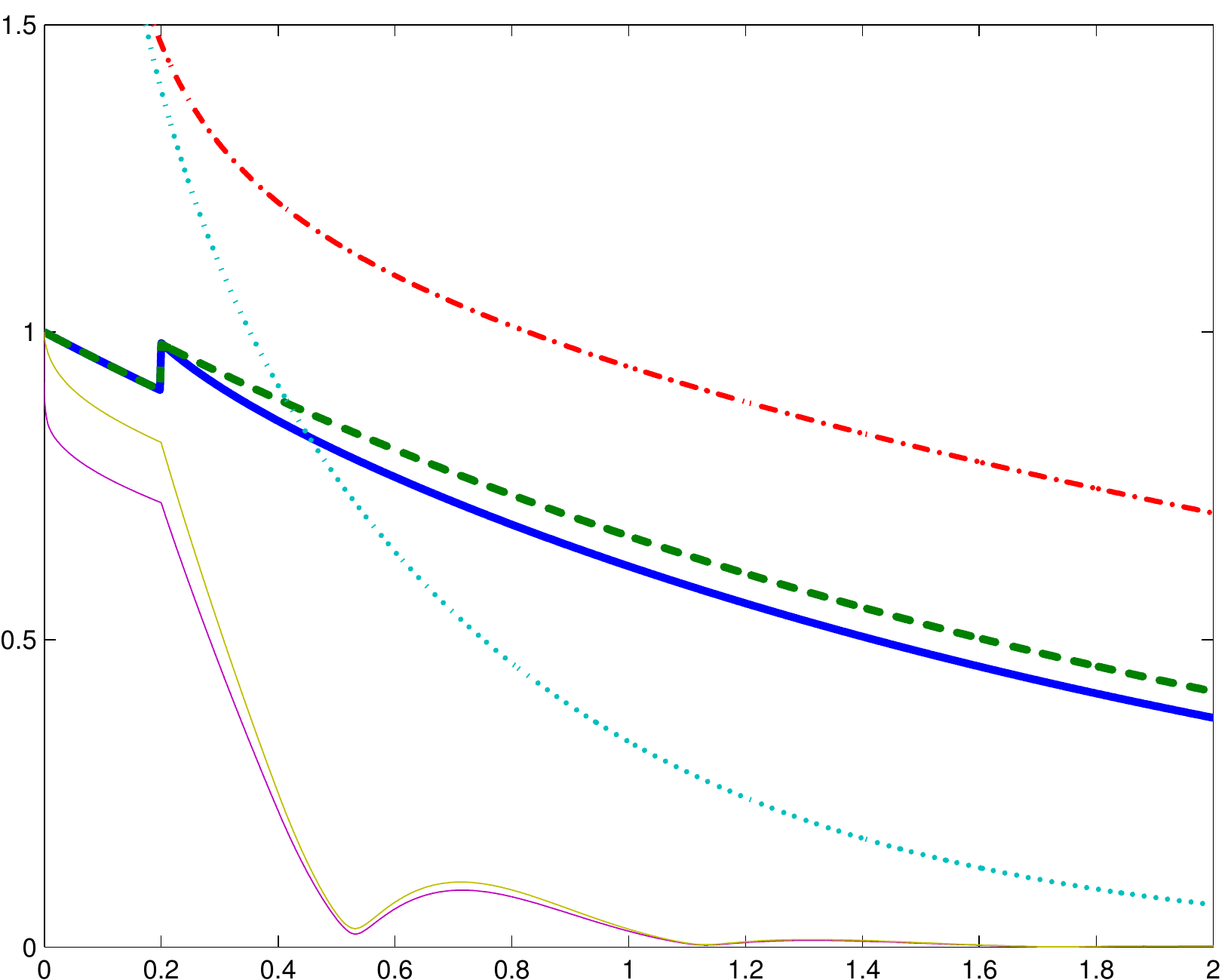}
 \caption{ Left:
           Inclusion regions for the eigenvalues of $T(z)$~\cite{bindelhood},
           the six largest eigenvalues of $A$,
           the contour $\Gamma$ used in Theorem~\ref{thm-DDEUB} (solid),
           and the vertical contour used for alternate bound~(\ref{eq-DDEUB-vert})
           (dashed).
           Contour parameters $y_0$ and $\eta$ are set to
           21.4214 and 0.0491366, respectively.
           For the contour with left-shifted vertical part (dotted),
           $y_0 = 27.7628$ and $\eta = 0.0491366$.
           Right:
           Upper bounds on the fundamental solution of~(\ref{eq-effex}):
           the upper bound described in the theorem (solid), and the
           alternates~(\ref{eq-DDEUB-vert}) (dashed),~(\ref{eq-DDEUB-nonsplit}) (dot-dashed),
           and~(\ref{eq-DDEUB-nonsplit}) with the
           contour with dotted vertical part (dotted).
           Lower solid curves are solutions to~(\ref{eq-effex}) with
           zero initial condition on $[-\tau,0)$ and various unit-norm initial
           conditions specified at $t = 0$.
         }
 \label{fig-PDDE-UB}
\end{figure}
In Figure~\ref{fig-PDDE-UB} (right), we have plotted the bound derived
using the theorem (solid), as well alternate bounds~(\ref{eq-DDEUB-vert})
(dashed),~(\ref{eq-DDEUB-nonsplit}) (dot-dashed),
and~(\ref{eq-DDEUB-nonsplit}) with the contour whose vertical part is shifted
leftward (dotted). The last bound gives better results for larger times,
as expected, but the bound given in Theorem~\ref{thm-DDEUB}
outperforms it for smaller
times and outperforms the other bounds for all times $t > \tau$.
\end{example}

In case $u([-\tau,0)) \not\equiv 0$, we can still obtain bounds
on $u(t)$ in terms of the fundamental solution $\Psi(t)$.
\begin{corollary}
\label{thm-DDEUBwithhist}
Suppose $u$ satisfies the DDE of the theorem subject
to initial conditions $u(0+) = u_0$ and $u(t) = \varphi(t)$ on
$[-\tau,0)$, with $B\varphi$ integrable. Then
\[
 \|u(t)\|_2 \le \|\Psi(t)\|_2\cdot \|u_0\|_2 +
              \sup_{0 \le \nu < \tau} \|\Psi(t-\nu)\|_2 \,
              \int_0^\tau \|B\varphi(\nu-\tau)\|_2 \,d\nu.
\]
Furthermore, this
bound is dominated by the more generous but
more explicit piecewise bound $k_1(t)\|u_0\|_2 +
k_2(t) \int_0^\tau \|B\varphi(\nu-\tau)\|_2 \,d\nu$, where
$k_1(t)$ is an upper bound on $\|\Psi(t)\|_2$ and
\[
 k_2(t) = \begin{cases}
           \displaystyle\sup_{0 < s < t} \|\exp(As)\|_2 & (0 \le t < \tau), \\
           \displaystyle\sup_{t-\tau < s < \tau} \|\exp(As)\|_2 +
           \displaystyle\sup_{\tau < s < t} \|\exp(As)\|_2 +
           e^{x_0\tau}(I_0 + C) & (\tau \le t < 2\tau), \\
           \displaystyle\sup_{t-\tau < s < t} \|\exp(As)\|_2 +
           e^{x_0(t-\tau)}\left(I_0 + \frac{C}{(t-\tau)/\tau}\right) & (t \ge 2\tau).
          \end{cases}
\]
\end{corollary}
\begin{proof}
 From~\cite[Ch. 1, Thm. 6.1]{hale1993},
 \[
  u(t) = \Psi(t)u_0 + \int_0^\tau \Psi(t-\nu)B\varphi(\nu-\tau)\,d\nu
 \]
 and the first assertion is obvious.
 Using the fact that $\Psi(t) = e^{At}$ on $0 \le t < \tau$ and $\Psi(t-\nu) = 0$ for
 $t-\nu < 0$, the integration path can be truncated to $\int_0^t$.
 Similarly, if $\tau \le t < 2\tau$, then $\Psi(t-\nu) = e^{A(t-\nu)}$ for $t-\tau < \nu < \tau$,
 and for this range of $\nu$ we have $t-\nu \ge \tau$. Finally, if $t \ge n\tau$ then
 $t-\nu \ge (n-1)\tau$ as $\nu \le \tau$. Taking norms and applying the theorem gives
 the piecewise bounds.
\end{proof}

\begin{example} \rm
We return to the model of the semiconductor laser with phase-conjugate feedback.
Since $A$ is not Hermitian, Theorem~\ref{thm-DDEUB} does not apply directly.
However, by letting $\beta$ equal the largest imaginary part of any eigenvalue
of $A$, changing the definition of $\Gamma_\infty$ so that
$x(y) = -\log\left(\eta(|y|-\beta)\right)$ instead, and
using the fact that $A = VDV^{-1}$ is diagonalizable, it is straightforward to
derive the same bound as in Theorem~\ref{thm-DDEUB} with the alteration
\[
 C = \frac{\kappa_2(V) \eta \|E\|_2 \sqrt{(\tau(y_0 - \beta))^{-1} + 1}}
          {\pi (1-\eta\|E\|_2)},
\]
where $E = V^{-1}BV$, $\kappa_2(V)$ is the 2-norm condition number of $V$, and
$y_0$ and $\eta$ were chosen to satisfy
\[
 1 < \eta(y_0-\beta) < \min\Bigg\lbrace \frac{y_0 - \beta}{\|E\|_2},
                                   e^{-\alpha(T)}, e^{-\alpha(A)}
                           \Bigg\rbrace.
\]
\begin{figure}[h]
 \includegraphics[width=0.48\linewidth]{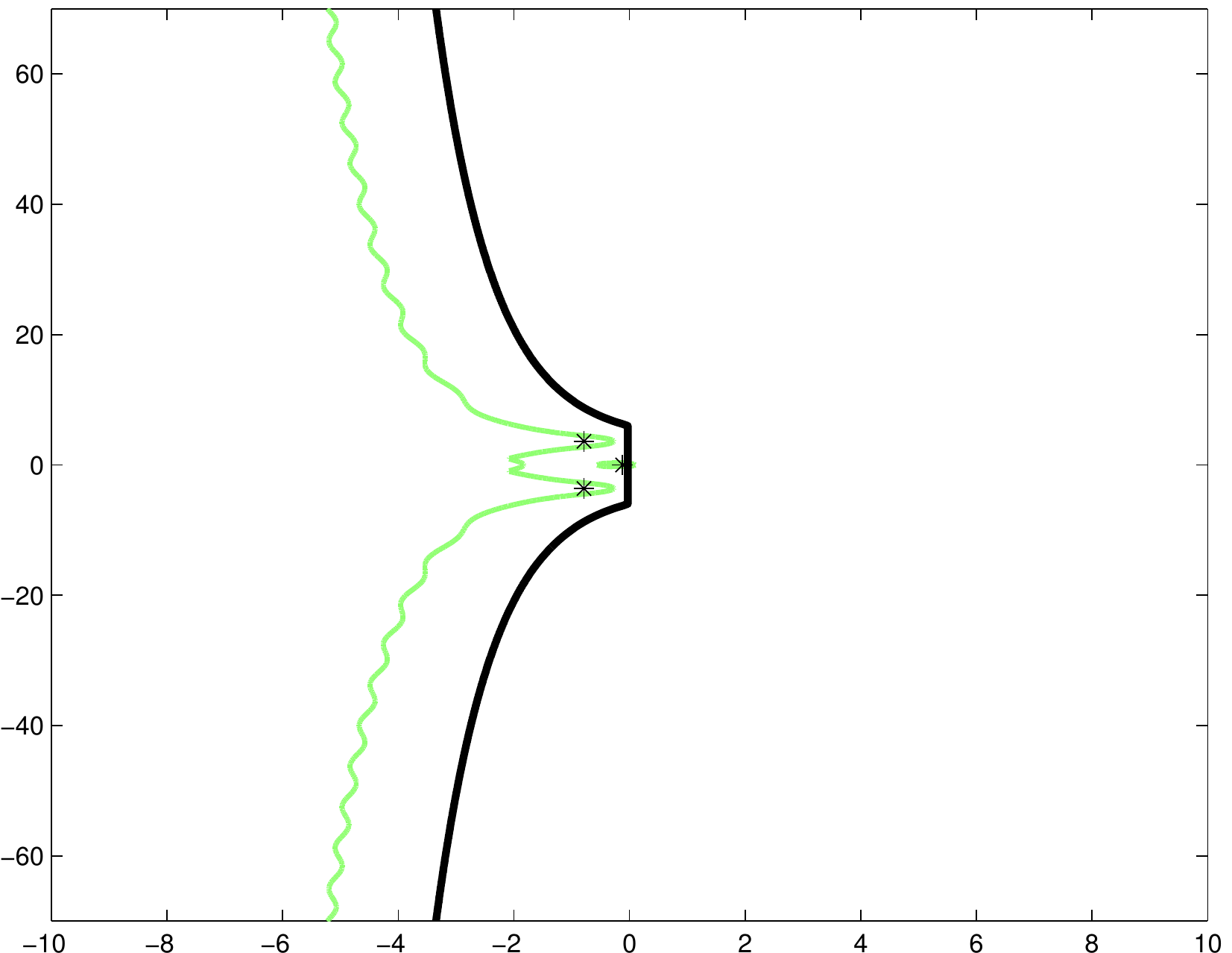}
 \hfill
 \includegraphics[width=0.48\linewidth]{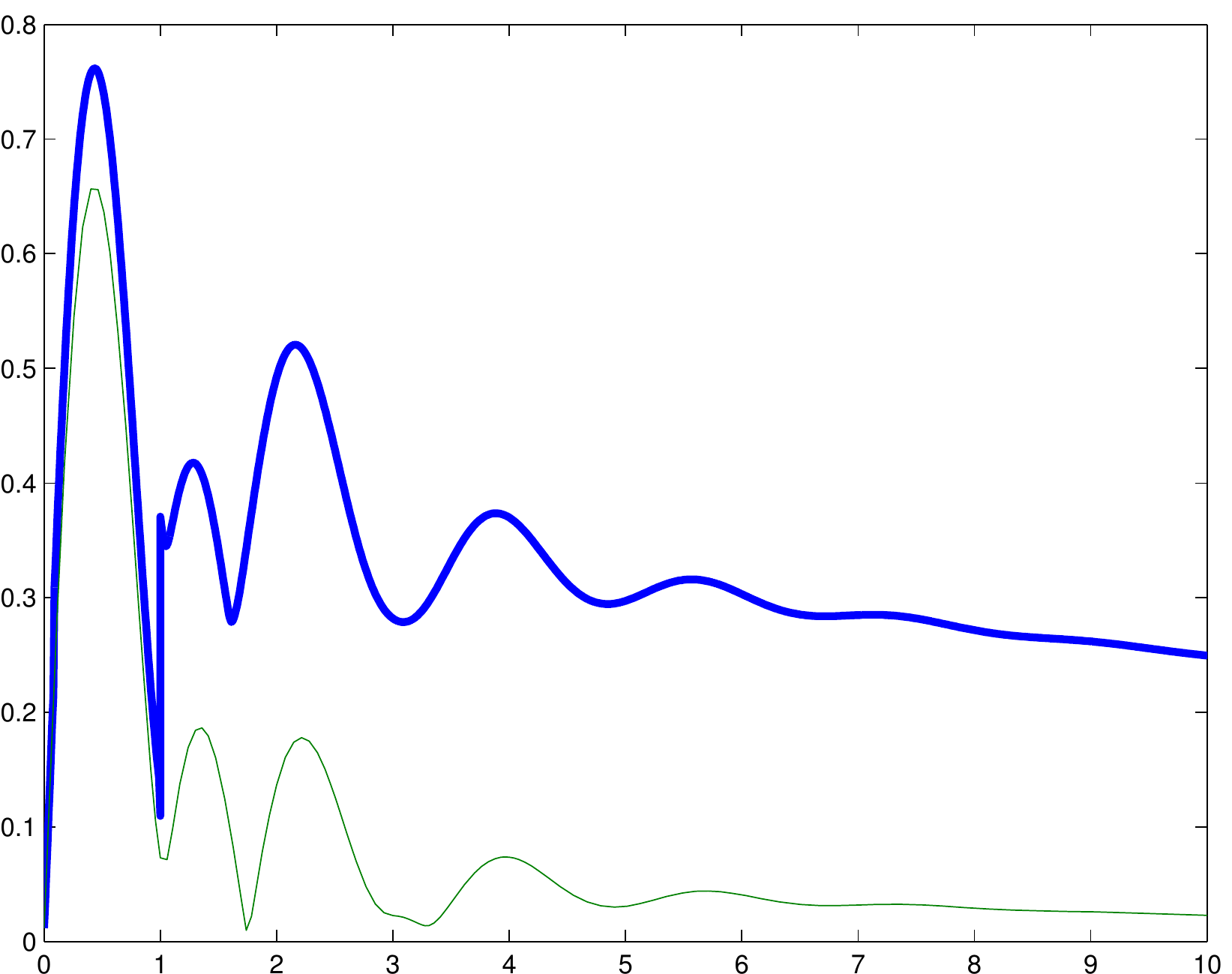}
 \caption{ Left:
           An inclusion region for the eigenvalues of $T(z)$~\cite{bindelhood},
           the three right-most eigenvalues of $T(z)$, and
           the contour $\Gamma$ (thick line).
           Contour parameters $y_0$ and $\eta$ are set to
           21.4214 and 0.0491366, respectively.
           Right:
           Upper bound on solution with initial conditions $\phi(t) \equiv u(0)$ (thick line) and the solution computed with Matlab's dde23.
         }
 \label{fig-linlaser-UB}
\end{figure}
 Inclusion regions were obtained for $\tilde{T}(z) = zI-\Lambda-Ee^{-z}$ with the
splitting $D(z) = zI-\Lambda-E_0e^{-z}$ and $E(z) = Fe^{-z}$ ($E_0 = \diag E$, $F = E-E_0$)
and an application of Theorem 3.1 in~\cite{bindelhood}. The one component
of the inclusion region intersecting the right half-plane contains exactly one
eigenvalue of $D(z)$, and therefore exactly one eigenvalue of $T(z)$,
which we have computed
using Newton's method on a bordered system~\cite[Chapter 3]{govaerts2000}.
\end{example}

\section{Lower bounds}
\label{sec-lower}

The following is an extension of
Theorem~\ref{thm-treflower}.
\begin{theorem}
\label{thm-genLB}
 Suppose we have the representations
 \[
  \Psi(t) = \frac{1}{2\pi i} \int_\Gamma T(z)^{-1} e^{zt}\,dz
  \qquad\text{and}\qquad
  T(z)^{-1} = \int_0^\infty \Psi(t) e^{-zt}\,dt
 \]
 where the latter holds for $\Re(z) > \alpha(T)$, $T$ some
 matrix-valued function, and $\Gamma$ is a (possibly unbounded)
 curve in the complex plane. Then for arbitrary $\omega \in \bbR$,
 \[
  \sup_{t \ge 0} \|\Psi(t)\| e^{-\omega t} \ge
  \frac{\alpha_\ve(T)-\omega}{\ve}
 \]
 for any $\ve > 0$ for which $\alpha_\ve(T)$ is finite.
\end{theorem}
\begin{proof}
 Suppose $\|\Psi(t)\| \le Ce^{\omega t}$ for all $t \ge 0$
 and fix $\ve > 0$.
 Without loss of generality, suppose that $\omega <
 \alpha_\ve(T)$, and take $z \in \Lambda_\ve(T)$
 such that $\Re z > \alpha(T)$ and $\Re z > \omega$.
 Then by the representation for $T(z)^{-1}$,
 the bound on $\Psi(t)$ implies
 $\|T(z)^{-1}\| \le \frac{C}{\Re(z) - \omega}$.
 Then $\|T(z)^{-1}\| > \ve^{-1}$
 by definition of $\Lambda_\ve(T)$, which
 implies $\Re(z) - \omega < C\ve$. It
 follows that $\alpha_\ve(T) \le C\ve + \omega$. By the
 contrapositive, $\alpha_\ve(T) > C\ve + \omega$ implies
 the desired result.
\end{proof}

The following proposition is easier to use in practice.
\begin{proposition}\label{thm-practicalLB}
 For each $x > 0$, $\sup_{t\ge 0} \|\Psi(t)\|
 \ge x \sup_{y \in \bbR}  \|T(x+iy)^{-1}\|$.
\end{proposition}
\begin{proof}
Fix $x$ and let $y \in \bbR$ be arbitrary.
Set $\|T(x+iy)^{-1}\| = \ve^{-1}$. Then $\alpha_\ve(T) \ge x$.
From Theorem~\ref{thm-genLB},
$\sup_{t\ge 0} \|\Psi(t)\| \ge x \ve^{-1}$.
For fixed $x$, the right-hand side is maximized by finding
$\ve$ as small as possible, i.e. by finding $y$ such
that $\|T(x+iy)^{-1}\|$ is as large as possible.
\end{proof}
\begin{remark}
 Note that for a given $x$ we may only need to check a finite
range of $y$ values. This can be shown by proving that for
$|y|$ sufficiently large $\|T(x+iy)^{-1}\| \le \|T(x)^{-1}\|$,
for example.
\end{remark}

\begin{example} \rm
 We now give lower bounds on worst-case growth for
 our two examples. In the case of the
 discretized PDDE, we use the fact that $A$ is Hermitian
 to derive $\|T(x+iy)^{-1}\|_2 \le \|T(x)^{-1}\|_2$ for
 $|y| > |x| + \|B\|_2 e^{-\tau x} + \sigma_{\min}(T(x))$ as
 per the previous remark, and check 100 equally spaced
 $x$ values in $[5,10]$ for the largest lower bound given by
 Proposition~\ref{thm-practicalLB}.
 For the linearization of the laser
 example, where $A$ is not Hermitian, we use instead that
 $\|T(x+iy)^{-1}\|_2 \le \|T(x)^{-1}\|_2$ for
 $|y| > \|xI-A\|_2 + \|B\|_2 e^{-x} + \sigma_{\min} (T(x))$ and
 check an equally spaced 100 point mesh of $[1,5]$.

\begin{figure}[h]
 \includegraphics[width=0.48\linewidth]{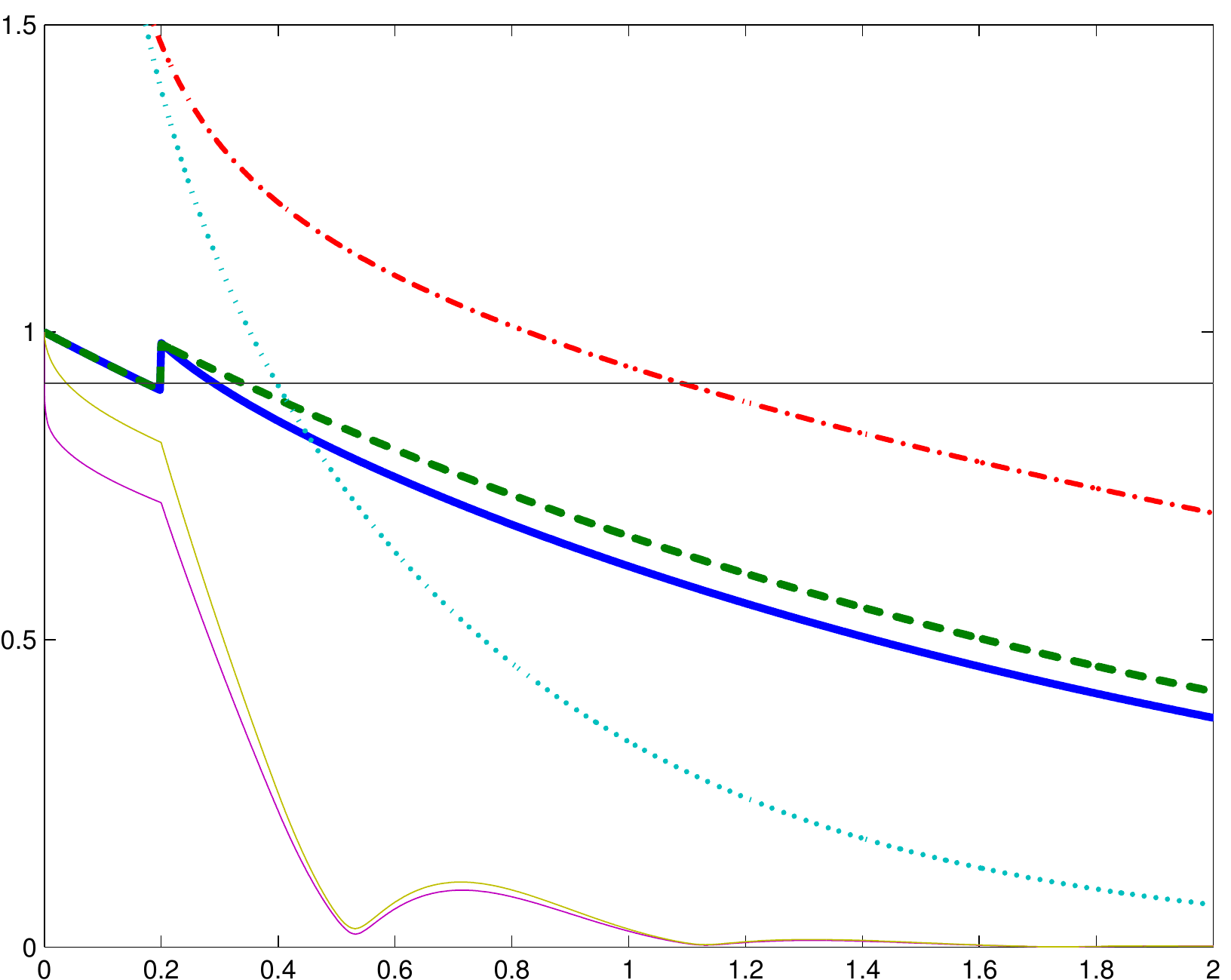}
 \hfill
 \includegraphics[width=0.48\linewidth]{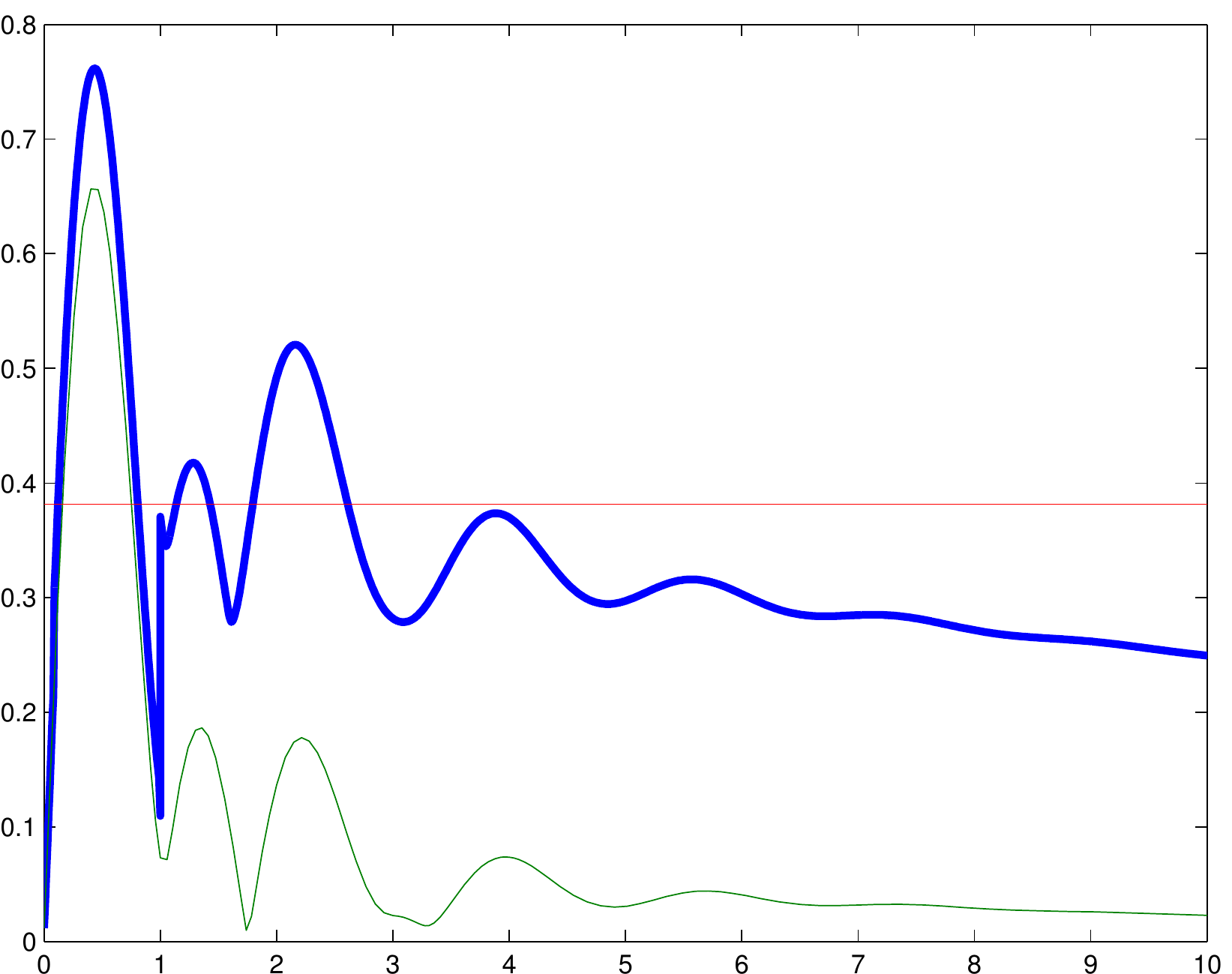}
 \caption{ Left:
           Discretized PDDE example with lower bound for solutions
           with $\|u(0)\|_2 = 1$.
           Notice that one but not both of the plotted solutions
           for the discretized partial DDE have supremum above
           the given lower bound.
           Right:
           Linearization of laser example with lower bound for
           solutions with $\|u(0)\|_2 = 0.0015\|(E_x,E_y,N)^T\|_2$.
           The solution to the linearized
           system from the laser model departs significantly from
           the equilibrium before decaying asymptotically. Unless
           the truncated nonlinear part of the original laser model
           is guaranteed to stay small under departures which
           differ from the equilibrium by 0.38242 in norm, the
           applicability of the linear stability analysis to
           this equilibrium and these initial conditions is
           questionable.
         }
 \label{fig-linlaser-UB}
\end{figure}
\end{example}

\section{Conclusion}
\label{sec-conc}
Some practical, pointwise upper bounds on
solutions to higher-order ODEs and single, constant
delay DDEs have been demonstrated on a discretized
partial DDE and a DDE model of a semiconductor
laser with phase-conjugate feedback. A general
lower bound was stated and used to concretely
bound worst-case transient growth for both
examples with a small computational effort.
Effective techniques for localizing eigenvalues
rather than computing them were used in an
auxiliary fashion.




\bibliographystyle{siam}
\bibliography{refs}

\end{document}